\documentclass[11pt,twoside]{article}
\usepackage{latexsym}
\usepackage{amssymb,amsbsy,amsmath,amsfonts,amssymb,amscd}
\usepackage{graphicx,color}
\usepackage{float,url}
\usepackage{cancel}
\usepackage[colorlinks,linkcolor=red,citecolor=blue]{hyperref}
\newcommand{\commentout}[1]{}
\newcommand{\R}{\mathbb{R}}
\newcommand{\N}{\mathbb{N}}

\newcommand {\al} {\alpha}

\newcommand {\sg} {\sigma}

\newcommand {\f}   {\frac}
\newcommand {\p}   {\partial}
\newcommand{\dis}{\displaystyle}

\newcommand{\beq}{\begin{equation}}
\newcommand{\eeq}{\end{equation}}
\newcommand{\bea} {\begin{array}{rl}}
\newcommand{\eea} {\end{array}}
\newcommand{\bepa}{\left\{ \begin{array}{l}}
\newcommand{\eepa} {\end{array}\right.}
\newtheorem{theorem}{Theorem}

\newtheorem{remark}[theorem]{Remark}
\newtheorem{proposition}[theorem]{Proposition}

\newcommand{\qed}{{ \hfill
                       {\unskip\kern 6pt\penalty 500 \raise -2pt\hbox{\vrule\vbox to 6pt{\hrule width 6pt
                       \vfill\hrule}\vrule} \par}   }}
\usepackage[dvipsnames]{xcolor}

\title{Bounds and long term convergence for the voltage-conductance kinetic system arising in  neuroscience}

\author{Xu'an Dou\thanks{Beijing International Center for Mathematical Research, Peking University, Beijing, 100871, China. Email : dxa@pku.edu.cn} \and
Beno\^ \i t Perthame\thanks{Sorbonne Universit{\'e}, CNRS, Universit\'{e} de Paris, Inria, Laboratoire Jacques-Louis Lions UMR7598, F-75005 Paris. 
Email : Benoit.Perthame@sorbonne-universite.fr}
\and 
Delphine Salort\thanks{Sorbonne Universit{\'e}, CNRS, Laboratoire de Biologie Computationnelle et Quantitative, UMR 7238, F-75005 Paris. Email : delphine.salort@sorbonne-universite.fr} 
\and 
Zhennan Zhou\thanks{Beijing International Center for Mathematical Research, Peking University, Beijing, 100871, China. Email : zhennan@bicmr.pku.edu.cn}
}
\date{\today}

\begin{document}
\maketitle
\pagestyle{plain}

{\em In honour of Juan-Luis Vazquez's 75th birthday}

\pagenumbering{arabic}

\begin{abstract} 
The voltage-conductance equation determines the probability distribution of a stochastic process describing a fluctuation-driven neuronal network arising in the visual cortex. Its structure and degeneracy share many common features with the kinetic Fokker-Planck equation, which has attracted much attention recently.

We prove an $L^\infty$ bound on the steady state solution and the long term convergence of the evolution problem towards this stationary state. Despite the  hypoellipticity property, the difficulty is to treat the boundary conditions that change type along the boundary. This leads us to use specific weights in the Alikakos iterations and adapt the relative entropy method.
\end{abstract} 
\vskip .7cm

\noindent{\makebox[1in]\hrulefill}\newline
\section{Introduction}  
\label{sec:intro}

A class of parabolic equations arise in the field of neuroscience to describe the voltage dynamics of assemblies of neurons. They determine the probability distribution of neurons according to various possible variables such as potential, current, conductance. In particular, in the context of  the primary visual cortex (V1), because of slow post-synaptic receptors, it has been  necessary to include the dynamics of conductance~\cite{CTRM06,cai2004effective, RKC}. This led to describe the neurons by the probability density  $p(v,g,t)$ to find neurons  at time $t$ with a membrane potential (voltage) $v$ and a conductance $g$ thanks to the equation (it is usually called the {\em voltage-conductance} equation)
\begin{equation}\label{eq:iftdp}
 \begin{split}
 & \frac{\partial}{\partial t}p(t,v,g)+    \frac{\partial}{\partial v} \left[\big( -g_L v  +g(V_E -v )\big) p(t,v,g)\right]
+ \frac{\partial}{\partial g} \left[ \f{ g_{\rm in} - g }{\sigma_{E}}  p(t,v,g)\right] \\
 \end{split}
\end{equation}
\begin{equation*}
 \begin{split}
& -  \frac{a}{\sigma_E}  \frac{\p^2}{\p g^2} p(t,v,g)  =  0,\quad t>0,\,v\in(0,V_F),\,g\in(0,+\infty).
 \end{split}
\end{equation*}
We refer to~\cite{CTRM06,cai2004effective} for the interpretation of the coefficients.  The firing and the excitatory reversal potentials satisfy $0 < V_F < V_E$ and we use $0$ as the leak  potential, $\sigma_E$ is the time delay constant of the excitatory conductance,  the leak conductance is $g_L >0$, $g_{\rm in} >0$ denotes the conductance induced by input currents and $a>0$ denotes the induced noise level in conductance. For simplicity of analysis, $g_{\rm in}$ and $a$ are  taken to be constant for the rest of the paper while the derivation of the model leads to  more elaborated, nonlinear and time dependent, formulas.
\\

For integrate and fire models, specific boundary conditions are included, which are presented below in a simpler mathematical formalism. Several studies have been devoted  to this equation, its  derivation and
applications,  in particular by~\cite{CCTao,rangan2006maximum,LyT} and to numerical solutions~\cite{RCT, CCTao}. An interesting feature is the behaviour of solutions which can be merely relaxation to a steady state, or oscillatory behaviour in the nonlinear case when $g_{\rm in}=g_{\rm in}(t)$ is defined through the solution $p(t,v,g)$, at least in certain ranges of the parameters as numerically observed in \cite{RCT}. See also a recent study towards understanding such a periodic behavior in \cite{CDZ2022}.
\\ 
 
 \paragraph{Main results.} 
In order to understand the behaviour of solutions, a first step is to study the steady state and to prove global regularity properties beyond the $L^{\f 8 7 -}$ estimate obtained in \cite{PeSa}. 
So we first consider here  the problem to find a solution $p(v,g)$ of  the stationary equation associated with~\eqref{eq:iftdp}, that is, for $(v,g) \in \Omega := (0, V_F) \times (0, \infty)$, we  consider the steady state problem 
\beq
\begin{cases}
\p_v [J_v(v,g) p(v,g)] - \frac{a}{\sigma_E} \p_g  \big[ M(g) \p_{g} [M^{-1}(g) p(v,g)] \big] =0,     \qquad \text{in } \Omega,
\\[5pt]
 p\geq 0, \qquad \int_\Omega p(v,g) dvdg=1,
\\[5pt]
\f{g-g_{\rm in}}{\sg_E} p + \frac{a}{\sigma_E}  \p_{g} p = 0, \quad at \quad g=0,
\\[5pt]
J_v\,  p (V_F,g)= J_v\,  p (0,g)=0, \quad g\leq g_F, \qquad \quad J_v \,p (0,g) = J_v \,p (V_F,g)   \quad g\geq g_F,
\end{cases}
\label{eq:cif}
\eeq 
with the  $v$-flux defined by
$$
J_v(v,g) = -g_Lv + g(V_E-v), \quad \text{and } \quad M(g) = e^{-\f{| g-g_{\rm in} |^2}{2a}}.
$$
Here, the At $g=0$, the boundary condition is simply the standard no-flux condition, and in $v$ it depends on the sign of the drift. At $v=0$ we notice that $J_v(0,g) >0$ and we always need to define an entering flux. At $V_F$, the flux sign   depends on $g$ through the value 
$$
g_F = \f{g_L V_F }{V_E-V_F }, \quad \text{characterized by }  J(V_F, g_F)=0.
$$
When $g < g_F$, then $J_v(V_F, g) <0$ and we impose zero entering flues at both $v=0$ and $v=V_F$. When $g>g_F$, then $J_v(V_F, g) >0$ and we set the entering flux at $v=0$ to be equal to the outgoing flux at $v=V_F$, i.e., $J_v(0,g)p(0,g)=J_v(V_F, g)p(V_F,g)$. Note that this flux equality holds for all $g>0$, which implies the conservation of mass for the corresponding evolution equation \eqref{eq:iftdp}. For interpretations of the boundary condition in a biology context, the readers may refer to \cite{CTRM06,PeSa}.
\\

Our first purpose is to improve the known Lebesgue regularity of the solution and prove 
\begin{theorem} [$L^\infty$ bounds or the stationary state]
The solution $p^*(v,g)$ of equation~\eqref{eq:cif} belongs to $L^\infty \big( (0, V_F)\times (0, \infty)\big)$.
\label{thm1}
\end{theorem}
Not only we prove the $L^\infty$ bound but the method, based on  the celebrated Alikakos iterations~\cite{alikakos79}, allows to establish quantitative estimates. This might be useful because asymptotic problems are relevant, see~\cite{PS2019, KPS2021} for instance and require specific bounds as discussed in Section~\ref{sec:cp}. 
\\

Our second purpose is to study the long time convergence of solutions of the evolution problem towards this steady state and to improve the convergence result for $v$ integrals in~\cite{PeSa}. Our result is also a worthy addition to previous works on the long time behavior of simplified models, \cite{dou2022exponential,CDZ2022}. We have
\begin{theorem} [Long term behaviour]
Assume an initial data $p^0 \leq C^0_+ p^*$ ($p^0$ is a a probability density), then the solution $p(t,v,g)$ of the evolution problem~\eqref{eq:iftdp} associated with equation~\eqref{eq:cif} satisfies $p(t)\leq C^0_+ p^*$ and 
\[
\int_\Omega \left( \frac{p(t,v,g)}{p^*(v,g)} -1\right)^2 p^*(v,g) dv dg \to 0, \quad \text{as } \; t \to \infty.
\]
\label{thm2}
\end{theorem}

The main difficulty comes from the lack of ellipticity in~\eqref{eq:cif}, which is short of a second order term in~$v$. The classical H\"{o}rmander theorem~\cite{hormander1967hypoelliptic, hypobookinvite,nier2005hypoelliptic} ensures that the differential operator in~\eqref{eq:cif} is hypoelliptic, which implies the \emph{interior} smoothness of its solution $p$. 
However, such a result is local in nature and to establish quantitative estimates, in particular near the boundary, seems a hard task.  See however~\cite{bernou2021} for the case of the kinetic Fokker-Planck equation.

\paragraph{Connected works.} Similar difficulties arise in several other problems but it seems that none of the methods applies to~\eqref{eq:cif}, in particular to the complex boundary condition in $v$.

The above equation shares common features with the classical Vlasov-Fokker-Planck equation which makes it also attractive for mathematical analysis. In particular because it satisfies the full rank Lie bracket property, it makes it attractive to use the recent methods  related to hypocoercivity \cite{Villani_hypocoercivity, DMS_2015, Herda_Rodrigues, Monmarche2014,BDMMS}. However several major obstacles arise in adapting the various methods proposed in the case of the kinetic Fokker-Planck equation. Among them are the specific structure of the voltage-conductance Integrate and Fire model, the boundary conditions, and from the start the properties of the steady state which are not explicitly available. A more detailed discussion on these difficulties is given \cite{dou2022exponential} where a simplified model is analyzed.

One can also view the steady equation \eqref{eq:cif} as a degenerate elliptic problem, since its second order term is given by $\frac{a}{\sigma_E}\p_{gg}p$. However, in spite of many existing results on degenerate elliptic problems (e.g. \cite{FicheraMR0111931,Oleinik,MR1077278,CattiauxMR1182641}), they do not quite apply to \eqref{eq:cif} due to two points. Firstly, it is usually required a sign condition on the zeroth order term, which is not satisfied in \eqref{eq:cif}. Secondly, most literature treats Dirichlet or Neumann type boundary conditions, which is quite different from the complicated boundary condition in $v$ of \eqref{eq:cif}. Therefore, to prove the $L^{\infty}$ estimate for the steady state, we need to exploit specific structures of \eqref{eq:cif} which goes beyond general results in literature.

\paragraph{Outline of the paper.}
The proof of Theorem~\ref{thm1} is given in the next two sections. We begin with establish a first integrability result in $L^{\frac 4 3 -}$ which allows us to introduce the multiplier compatible with the boundary conditions. Then, in Section~\ref{sec:3}, we show how to iterate the estimate and prove the $L^\infty$ bound. Then, Section~\ref{sec:ltc} is devoted to the long term convergence. We begin with recalling the relative entropy bounds which provide the main tool to prove Theorem~\ref{thm2}. To conclude, we mention several open problems.

\section{The weighted  $L^{\f 43-}$ estimate revisited} \label{sec:2}

We revisit and simplify the method from~\cite{PeSa} for higher integrability and present the main result  of this section in  Proposition~\ref{propL43} below. We present  the multipiers which serve as the foundation for the iterating scheme to be introduced in Section~\ref{sec:3}. As preparations, we recall some basic estimates for the steady state satisfying \eqref{eq:cif}.

By directly integrating the equation in $v$, we see that the $g$-marginal of the steady state is a Gaussian, given by
\beq 
\int_0^{V_F} p(v,g) dv =Z^{-1} M(g) ,
\label{est:1}
\eeq where $Z:=\int_0^{+\infty}M(g)dg$ is a normalization constant. Another immediate observation is the following regularity in $g$ 
\beq\label{est:1.5}
\int_0^{V_F}\hskip-5pt  \int_0^\infty |\p_gp(v,g)|^2\frac{dg dv}{p(v,g)} <+\infty
\eeq
 This estimate \eqref{est:1.5} is proved  by multiplying the equation \eqref{eq:cif} with $\ln p$ and integrating over $v$ and $g$. Combining \eqref{est:1.5} with the decay property of the Gaussian in \eqref{est:1}, one has 
\beq 
\int_0^{V_F}\hskip-5pt  \int_0^\infty e^{\f{g^2}{8a}} |\p_gp(v,g)|dg dv \leq K_1.
\label{est:2}
\eeq
The interested readers can refer to \cite[Section 3.2]{PeSa} for the detailed proofs of \eqref{est:1.5} and \eqref{est:2}.

Next, we introduce our improved perspective to show the higher integrability of the steady state, by which we simplify the approach in~\cite{PeSa} and establish  the following estimates 
\begin{proposition} With $K_1$ defined in \eqref{est:2}, we have
\beq 
\int_0^{V_F} \sup_g  p(v,g) dv \leq K_2:=K_1,
\label{est:3}
\eeq
\beq 
\sup_v \int_0^{\infty}  J_v^2 p(v,g) dg \leq K_3, 
\label{est:4}
\eeq
\beq 
\int_0^{V_F} \hskip-5pt  \int_0^{\infty}  | J_v \, p(v,g)|^2  dg dv  \leq  K_2 K_3 ,
\label{est:5}
\eeq
\beq 
\int_0^{V_F} \hskip-5pt  \int_0^{\infty}   (g+g_{L})^2 p(v,g)^q  dg dv < \infty  \qquad \forall q < \f 4 3 .
\label{est:6}
\eeq
\label{propL43}
\end{proposition}
We point out that a major issue in~\eqref{est:3}, compared to~\eqref{est:4} is to establish that the $sup$ is inside the integral. This allows us to by-pass  the Besov imbedings used in~\cite{PeSa}. 
\\
\begin{proof}
To prove \eqref{est:3}, we write
\[
p(v,g) =-\int_g^\infty \p_g p(v,g') dg' \leq \int_0^\infty |\p_g p|(v,g') dg',
\]
which also gives
\[
\int_0^{V_F}\sup_g p(v,g) dv \leq \int_0^{V_F} \int_0^\infty |\p_g p|(v,g) dg \leq K_1.
\]

To prove \eqref{est:4}, we write
\[
\p_v [J_v^2 p]-J_v p  \p_v J_v - \frac{a}{\sigma_E} J_v\p_g \big[M(g) \p_{g} [pM^{-1}(g)] \big]=0,
\]
\beq\label{est-J2p}
\p_v \int_0^\infty J_v^2 p dg + \int_0^\infty (g +g_L) J_v p dg + \frac{V_E-v}{\sigma_E} \int_0^\infty \underbrace{ aM(g) \p_{g} [pM^{-1}(g)]}_{(g-g_{in})p + a\p_g p} dg=0,
\eeq
which means that $\p_v \int_0^\infty J^2\,  p dg$ is controlled in $L^1_v$ and since $\int_0^\infty J^2 p dg$ is also controlled in $L^1_v$, it gives~\eqref{est:4}.

Then, to prove \eqref{est:5}, we combine \eqref{est:3} and \eqref{est:4}, and find 
\[
\int_0^{V_F}  \hskip-5pt  \int_0^{\infty}  J_v^2 \, p(v,g)^2 \leq \int_0^{V_F}  \int_0^{\infty} J_v^2 p(v,g )  \sup_{g'} [p(v,g')] dv dg
\]
\[
\leq  \sup_v  \int_0^{\infty} J_v^2 p(v,g )dg \; \int_0^{V_F} \sup_{g'} [ p(v,g')] dv.
\]

Finally, to prove \eqref{est:6}, thanks to \eqref{est:3}, we find immediately that for all $G>0$,  and $0\leq \al < 1$(Lemma~3 of~\cite{PeSa}) 
\beq 
\int_0^{V_F}\hskip-5pt  \int_0^G \f{ p(v,g)}{|J_v|^\al} dg dv \leq \int_0^{V_F} \sup_{g'}  p(v,g')dv   \int_0^G \f{ 1}{|J_v|^\al} dg \leq CK_2 .
\label{est:7}
\eeq
Then, on the one hand, for $q= \f 2 r +\f 1 {r'}$ and $\f 2 r =\f {\al} {r'}$ (e.g., $\al=1$, $r=3$, $q=\f 4 3$), 
\[
\int_0^{V_F} \hskip-5pt  \int_0^G  p^ q dg dv = \int_0^{V_F}\hskip-5pt \int_0^G (J_v^2 p^2)^{\f 1 r} \left( \f p {|J_v|^{\al}} \right)^{\f 1 {r'}}  \leq  \left(\int_0^{V_F} \hskip-5pt \int_0^G J_v^2 p^2\right)^{\f 1 {r}}     \left( \int_0^{V_F} \sup_g p   \int_0^G \f 1 {|J_v|^{\al}}  dg dv \right)^{\f 1 {r'}} 
\]
but the limiting case $\al=1$ is not allowed which gives the strict inequality $q<\frac{4}{3}$.
On the other hand, for $G>g_F$ and with $r=3$ as above, we have using \eqref{est:5} and \eqref{est:2}
\[
\int_0^{V_F} \hskip-5pt  \int_G^\infty  (g+g_{L})^2   p^{\f 4 3} dg dv \leq C  \left(\int_0^{V_F} \hskip-5pt  \int_G^\infty J_v^2  p^2 dg dv  \right)^{\f 1 {r}}   \left(\int_0^{V_F} \hskip-5pt  \int_G^\infty (g+g_L)^2  p dg dv  \right)^{\f 1 {r'}}  < \infty.
\]
\end{proof} \qed

In Proposition \ref{propL43}, we utilize the information in $v$ and $g$ direction respectively in \eqref{est:3} and \eqref{est:4}, a combination of which leads to the $L^2$ estimate for the flux $J_vp$ \eqref{est:5}. This is more direct than the way in \cite{PeSa}, which gains integrability through a Besov injection. Same as \cite{PeSa}, to get a higher integrability for $p$, we have to estimate the flux $J_vp$ first.

To recover the integrability of $p$ from \eqref{est:5}, we need to overcome the singularity due to zeros of $J_v=-g_Lv+g(V_E-v)$. This difficulty is treated in the same way as in \cite{PeSa}. Specifically, note that $J_v(v,g)$ has a positive lower bound if $g\geq G>g_F$, therefore the singularity only shows up for $g$ in a bounded interval $[0,G]$. Then we exploit the linearity of $J_v$, which implies that for a fixed $v$, $|J_v|^{-\alpha}$ is locally integrable in $g$ for $0\leq \al<1$. Hence, together with \eqref{est:3} we prove the integrability of $p/|J_v|^{\alpha}$ on $[0,V_F]\times[0,G]$, which allows us to recover the integrability by H\"{o}lder's inequality. The cost is, that the Lebesgue index is reduced from $q=2$ to $q<\frac{4}{3}$.

We remark that the weight in \eqref{est:4} is carefully chosen since $J_v$ can change sign. We also note that besides the integrability, \eqref{est:6} also gives some moment control in $g$, which is handy for the iteration in the next section. 

\section{Iterating and proof of Theorem \ref{thm1}}
\label{sec:3}

We are now ready to prove Theorem~\ref{thm1},  i.e. the $L^{\infty}$ estimate.

First, we extend the estimation methodology in Section \ref{sec:2} to $L^q,q>1$, which allows us to iterate. Starting from $1<q<\frac{4}{3}$ as in Proposition \ref{propL43}, we can gain higher integrability for all $1<q<\infty$.

To get the $L^{\infty}$ estimate, our strategy is to do the iterations in a more quantitative way. For heuristic purposes, let $a_q:=\|p\|_{L^q}^q$, and we aim to derive estimates of the form
\beq\label{example-est}
a_{2q}\leq C(q)a_q^2,
\eeq where $C(q)$ is a constant depending on $q$. The observation is that, as long as the growth of $C(q)$ in $q$ is at most polynomial, we can derive the $L^{\infty}$ estimate by iterating \eqref{example-est}. Specifically for a given $q$, the sup-limit $$\limsup_{n\rightarrow\infty}(a_{2^nq})^{{1}/{2^nq}}=\limsup_{n\rightarrow\infty}\|p\|_{L^{2^nq}}$$ would be finite, which is elementary to check. 

In the proof, we actually use a weighted version $d_q$ to be defined in \eqref{B:1} instead of $a_q$. And instead of linking $2q$ with $q$ as in \eqref{example-est}, we link $\beta q$ with $q$ where $1<\beta<\frac{4}{3}$ but the essence is the same. 

 With the scheme we have recalled, we now show the iteration procedure, which consecutively justifies higher Lebesgue exponents. And we conclude the proof of the $L^\infty$ estimate in Section~\ref{sec:32}  with the usual argument of Alikakos~\cite{alikakos79}.

\subsection{Iterating}

Departing with $q<\f 4 3$, as established in Proposition~\ref{propL43}, we are going to establish higher Lebesgue exponents for $p(v,g)$, using the notations 
\beq
d_q :=  \int_0^{V_F} \hskip-5pt  \int_0^{\infty}   (g+g_{L})^2 p(v,g)^q  dv dg < \infty .
\label{B:1}
\eeq
Then, we establish that
\begin{proposition} We have for $q>1$ (close to $\frac 4 3$ for the first iteration),
\beq 
q\int_0^{V_F} p(v,0)^q dv +\int_0^{V_F} \hskip-5pt  \int_0^{\infty}  | \p_g p(v,g)^{\f q 2}|^2  dg dv  \leq C q \, d_q , 
\label{B:2}
\eeq
\beq
\int_0^{V_F} \hskip-5pt  \int_0^{\infty} (g+g_L)   | \p_g p(v,g)^{\f q 2}|^2  dg dv  \leq C q \, d_q ,
\label{B:2a}
\eeq
\beq 
\int_0^{V_F} \sup_g  p(v,g)^q dv \leq C q^{\f 12}\, d_q ,
\label{B:3}
\eeq
\beq 
\sup_v \int_0^{\infty}  J_v^2 p(v,g)^q dg \leq C q\, d_q ,
\label{B:4}
\eeq
\beq 
\int_0^{V_F} \hskip-5pt  \int_0^{\infty}  | J_v \, p(v,g)^q|^2  dg dv  \leq  C q^{\frac 32}\, (d_q)^2 ,
\label{B:5}
\eeq
\beq 
\int_0^{V_F} \hskip-5pt  \int_0^{\infty}   (g+g_{L})^2 p(v,g)^{\beta q}  \leq C q^\beta\, (d_q)^\beta ,  \qquad \forall 0< \beta < \f 4 3 .
\label{B:6}
\eeq
\end{proposition}

\begin{proof}
To prove \eqref{B:2}, we mutiply equation~\eqref{eq:cif} by $p^{q-1}$ and find
\[
\p_v \f{J_vp^q}{q} - \f{q-1}{q}p^q\,  \left(\frac{1}{\sigma_E}+g+g_L\right) -\p_g \left[ \frac{g-g_{\rm in}}{\sigma_E} \f{p^q}{q} +\frac{a}{\sigma_E}\p_{g} \f{p^q}{q} \right]+ \frac{a}{\sigma_E} \f{4(q-1)}{q^2} \left(\p_g p^\f{q}{2} \right)^2=0 .
\]
It remains to integrate in $(v,g)$ and notice that the $v$ boundary conditions  give a non-negative contribution, since as before for $g>g_F$
\[
J_vp^q(V_F)- J_vp^q(0) = J_vp(V_F) p^{q-1}(V_F) \left[1- \f{p^{q-1}(0)}{p^{q-1}(V_F)}\right] = J_v p^q(V_F)\left[1- \left(\f{J_v(V_F,g)}{J_v(0,g)} \right)^{q-1}\right] \geq 0,
\]
as well as the zero-flux condition is $g$ because, for all $v$, at $g=0$ we may write
\beq\label{flux-g}
\frac{-g_{\rm in}}{\sigma_E} \f{p^q}{q} +\frac{a}{\sigma_E}\p_{g} \f{p^q}{q} =p^{q} (1-\frac 1 q) \frac {g_{\rm in}}{\sigma_E} \geq 0.
\eeq
The second inequality~\eqref{B:2a}, follows by multiplying the  above identity on $p^q$  by $g+g_L$. After integrating by parts and using the signs of the boundary terms, we obtain 
\begin{align*}
\frac{a}{\sigma_E} \f{4(q-1)}{q^2}\int_0^{V_F} \hskip-5pt  \int_0^{\infty}  (g+g_L)  \left(\p_g p^\f{q}{2} \right)^2 \leq &\int_0^{V_F} \hskip-5pt  \int_0^{\infty}  (g+g_L) p^q\,  \left(\frac{1}{\sigma_E}+g+g_L\right)
- \int_0^{V_F} \hskip-5pt  \int_0^{\infty}\frac{g-g_{\rm in}}{\sigma_E} \f{p^q}{q}
\\
&  -\frac{a}{\sigma_E}\int_0^{V_F} \hskip-5pt  \int_0^{\infty} \p_g(p^q)dvdg.
\end{align*}
The last term can be controlled from the boundary term in~\eqref{B:2} 
\[
-\int_0^{V_F} \hskip-5pt  \int_0^{\infty} \p_g(p^q)dvdg= \int_0^{V_F} p^q(v,0)dv\leq C d_q .
\]
And altogether, one finds 
\[
 \f{4(q-1)}{q^2} \int_0^{V_F} \hskip-5pt  \int_0^{\infty}  (g+g_L) \left(\p_g p^\f{q}{2} \right)^2 \leq C d_q.
\]

Then \eqref{B:3} is proved as before. We use twice the Cauchy-Schwarz inequality departing from
\[
p^q(v,g) =- \int_g^\infty \p_g p^q(v,g') dg' \leq 2\left(  \int_0^\infty |\p_g p^{\f q 2}|^2 (v,g') dg' \int_0^\infty p^q  (v,g') dg' \right)^{\f 1 2 }.
\]

For \eqref{B:4}, we write
\[
\p_v \f{J_v^2 p^q}{q} - J_v p^q\,  \left[ \f{q-1}{q \sigma_E}+(g+g_L)  \f{q-2}{q}\right] - J_v \p_g \left[ \frac{g-g_{\rm in}}{\sigma_E} \f{p^q}{q} + \frac{a}{\sigma_E}\p_{g} \f{p^q}{q} \right]+ \frac{a}{\sigma_E} J_v\f{4(q-1)}{q^2} \left(\p_g p^\f{q}{2} \right)^2=0 .
\]
Integrating in $g$ gives, for all $v\in (0,V_R)$,
\begin{align*}
\p_v\int_0^{\infty}   \f{J_v^2 p^q}{q} dg = &\int_0^{\infty}   J_v p^q\,  \left[ \f{q-1}{q \sigma_E}+(g+g_L)  \f{q-2}{q}\right] dg- \frac{a}{\sigma_E} \int_0^{\infty}  J_v\f{4(q-1)}{q^2} \left(\p_g p^\f{q}{2} \right)^2 dg 
\\
& - (V_E-v)  \int_0^{\infty}  \left[ \frac{g-g_{\rm in}}{\sigma_E} \f{p^q}{q} + \frac{a}{\sigma_E}\p_{g} \f{p^q}{q} \right]dg  + g_Lv(1-\frac 1 q) \frac {g_{\rm in}}{\sigma_E}p^{q}(v,0),
\end{align*}
 where the last term comes from the flux at $g=0$ as in \eqref{flux-g} and can be controlled by \eqref{B:2}.

Therefore, with a constant $C$ independent of $q$, 
\[
\left \| \p_v\int_0^{\infty}   J_v^2 p^q dg \right \|_{L^1(0, V_F)} \leq C q \, d_q + C  \int_0^{V_F} \hskip-5pt  \int_0^{\infty}  (g+g_L) \left(\p_g p^\f{q}{2} \right)^2 dg dv,
\]
which, thanks to \eqref{B:2} and the boundness  of $ \int_0^{V_F} \hskip-5pt  \int_0^{\infty} J_v^2 p^q dg dv$ gives  \eqref{B:4}. 
\\

Next, we recall that  \eqref{B:5} follows from  \eqref{B:3} and  \eqref{B:4} writing 
\[
\int_0^{V_F} \hskip-5pt  \int_0^{\infty}  | J_v \, p(v,g)^q|^2  dg dv  \leq \int_0^{V_F} \sup_g  p(v,g)^q dv \, \sup_v \int_0^{\infty}  J_v^2 p(v,g)^q dg \leq C q^{\frac{3}{2}} \, (d_q)^2.
\]

Finally, we conclude \eqref{B:6} as before because, on the one hand, for $G > g_F$, we have, always with $\beta q = \f{2q}{r}+ \f{q}{r'}$,  $\f 2 r =  \f \alpha  {r'}$, $0< \al <1$ (and $\beta \lessapprox \f 4 3 $ means $r  \gtrapprox 3$) 

\begin{align*}
\int_0^{V_F} \hskip-5pt  \int_0^{G} p(v,g)^{\beta q}  dg dv  &\leq \left( \int_0^{V_F} \hskip-5pt  \int_0^{G}  | J_v^2 p(v,g)^{2q}| dg dv\right)^{\f 1 r } 
\left( \int_0^{V_F} \hskip-5pt  \int_0^{G}   \f{p(v,g)^{q}}{|J_v|^\alpha}    \right)^{\f 1 {r'} } 
\\
& \leq \left( \int_0^{V_F} \hskip-5pt  \int_0^{G}  | J_v^2 p(v,g)^{2q}| dg dv\right)^{\f 1 r } \left( \int_0^{V_F} \sup_g p(v,g)^{q} dv \, \int_0^{G}   \f{1}{|J_v|^\alpha}  dg  \right)^{\f 1 {r'} } 
\\
& \leq  C q^\beta  (d_q)^\beta .
\end{align*}
And on the other hand, we have 
\begin{align*}
\int_0^{V_F} \hskip-5pt  \int_{G}^\infty (g+g_L)^2  p(v,g)^{\beta q}  dg dv  &\leq C \left( \int_0^{V_F} \hskip-5pt   \int_{G}^\infty  | J_v^2 p(v,g)^{2q}| dg dv\right)^{\f 1 r } 
\left( \int_0^{V_F} \hskip-5pt  \int_{G}^\infty  (g+g_L)^2 p(v,g)^{q}   \right)^{\f 1 {r'} } 
\\
& \leq  C [q^2   (d_q)^2]^{\f 1 r }  (d_q)^{\f 1 {r'} }  .
\end{align*}
Notice that the dependency on the exponent $q$ can be traced more accurately at the expense of more complex expressions, but the final results will not be altered.
\end{proof} \qed

\subsection{The $L^\infty$ bound} \label{sec:32}

We are now ready to conclude the proof of Theorem~\ref{thm1}.
\\

Iterating, we write for any $k\in \N$, $k>1$, $q=\beta^k$ and still with  $\beta \lessapprox \f 4 3 $,
\[
d_{\beta^{k+1}} \leq C  \beta^{k \beta}   d_{\beta^{k}}^\beta,
\]
\[
\big( d_{\beta^{k+1}} \big)^{1/\beta^{k+1}} \leq \big( C  \beta^{k \beta}\big)^{1/\beta^{k+1}}   \big( d_{\beta^{k}}^\beta\big)^{1/\beta^{k+1}} = \big( C_1  \beta^{k}\big)^{1/\beta^{k}}   \big( d_{\beta^{k}}\big)^{1/\beta^{k}}.
\]
In other words, we have found that  for all $K$
\[
\| p \|_{L^{\beta^{K}}} \leq \dis \prod_{k=1}^{K-1}   \big( C_1  \beta^{k}\big)^{1/\beta^{k}}   \;   \| p \|_{L^{\beta}} \leq C_2 \| p \|_{L^{\beta}} ,
\]
because the sequence 
\[
\ln  \prod_{k=1}^{K-1}   \big( C_1  \beta^{k}\big)^{1/\beta^{k}} =\sum _{k=1}^{K-1} \f{\ln C_1 + k\; \ln \beta}{\beta^{k}} 
\]
converges. We conclude that 
\[
\| p \|_{L^\infty} \leq C_2 \| p \|_{L^{\beta}} ,
\]
and the bound is proved.

This completes the proof of Theorem \ref{thm1}.
\qed

\section{Long term convergence} \label{sec:ltc}

The knowledge of properties of the steady state $p^*$ guides us to use relative entropy methods for the evolution problem~\eqref{eq:iftdp}. Here we explain the relative entropy estimate, together with a compactness method~\cite{RefMMP}, in order to prove the long term convergence to the steady state, still in the linear case. In the following, we choose $a=\sigma_E=1$ to simplify notations. 

Also, we assume that the unique steady state solution $p^*$ of~\eqref{eq:cif}, normalized as a probability, is positive and smooth away from the boundary $g \in (0,g_F)$ with $v=0$ or $v=V_F$. Thanks to the hypoellipticity of the operator, these properties follow from the maximum principle for degenerate equations \cite{bony1969principe,hill1970sharp}.

\subsection{Relative entropy}

We define the functions
$$
h(t,v,g):= \f{p(t,v,g) }{p^*(v,g)}, \qquad h^0(v,g):= \f{p^0(v,g) }{p^*(v,g)}\in L^\infty \big((0,V_F)\times (0,\infty) \big),
$$
where $p^0$, a probability density, denotes the initial data for~\eqref{eq:cif}.  

It is immediate to compute that $h$ satisfies the following equation (in strong form)
\begin{equation}\label{eq:h}
	\p_th+J_v\p_vh+\left(g_{\rm in} - g\right)\p_gh=[\p_{gg}h+\f{2}{p^*}\p_gp^*\p_gh], \qquad t > 0, \; (v,g)\in \Omega,
\end{equation} 
with the following boundary conditions for $t>0$
\begin{equation}
	h(t,0,g)=h(t,V_F,g),\quad  g>g_F,\qquad \p_g h(v,0)=0, \quad v\in(0,V_F).
\end{equation}
For a convex $C^1(\R; \R)$ function $H$, the chain rule gives
\begin{align*}
	\p_t H(h)+J_v\p_v H(h)&+\left(g_{\rm in} - g\right)\p_g H(h)=[\p_{gg}h+\frac{2}{p^*}\p_g p^*\p_g h]H'(h) \notag
	\\
	&= \p_{gg}H(h) +\frac{2}{p^*}\p_g p^*\p_g H(h) - H''(h(t,v,g))|\p_gh(t,v,g)|^2. 
\end{align*}
Notice that this is the same equation than for $h$ with an additional term on the right hand side. Therefore, when we 
multiply by $p^*$ and integrate on $\Omega$, we deduce
\begin{equation}\label{entropy-key}
  \frac{d}{dt}\int_{\Omega} H(h(t,v,g))p^*(v,g)dvdg=-\int_{\Omega} H''(h(t,v,g))|\p_gh(t,v,g)|^2p^*dvdg\leq 0.
\end{equation}
The integral on the left hand side is often referred to as a relative entropy. As is well-known, this dissipation gives us controls of $h$ or $p$, see \cite[Chapter 8]{perthame2015parabolic}. For instance we  find non-increasing weighted Lebesgue norms  $L^q$ ($1\leq q\leq \infty$). In particular we get a priori estimate in $L^{\infty}_t \big(L^{2}(\Omega)\big)$. These estimates can be translated back to $p$ under the form (here the $C^0_\pm$ are the initial controls)
\begin{equation}
C^0_- \,p^* \leq p(t) \leq C_+^0 \, p^*.
\label{hinfini}
\end{equation}
Therefore the  $L^{\infty}$ bound on $p^{*}$ of Theorem~\ref{thm1}, also provides us with a $L^{\infty}$ bound on $p$.

\subsection{Long term convergence}

Although the dissipation is not coercive, we can get convergence following a compactness method in~\cite[Section 3.6]{BP07}. The long term behaviour is replaced by understanding the limit as $k \to \infty$ of the sequences
\beq\label{def-pkhk}
p_k(t,v,g):=p(t+k,v,g), \qquad h_k(t,v,g):=h(t+k,v,g).
\eeq
From the bounds \eqref{hinfini}, we know that $h_k$ is bounded in $L^\infty\big((0,\infty)\times \Omega\big)$ and thus we may extract subsequences such that 
\[
h_{k(n)}(t,v,g) \rightharpoonup h_\infty, \quad \text{as } n\to \infty.
\]
To prove that the limit of $p_k$ is $p^*$ is equivalent to prove that $h_\infty =1$. This is what we establish in the following
\begin{proposition}
Let $p^*$ be the unique steady state. Assume $h^0 \in L^\infty (\Omega)$, and there is a constant $C^1_+$ such that
\begin{equation}\label{as:init-h}
\big|\p_v [J_v(v,g) p^0(v,g)] - a\p_g  \big[ M(g) \p_{g} [M^{-1}(g) p^0(v,g)] \big] \big| \leq C^1_+ p^*.
\end{equation}
Then, $h_\infty=1$, i.e., the sequence $p_k$ converges to the unique steady state $p^*$ and the convergence is almost everywhere and in all $L^q_{loc}\big((-T,T)\times \Omega \big)$ with $ 1 \leq q < \infty$.
\label{prop5}
\end{proposition}

\begin{proof}
{\em Step 1. Entropy dissipation.}
We choose that $H(h)=(h-1)^2$ then $H''(h)=2$. From \eqref{entropy-key} we have
\begin{equation}
\frac{d}{dt}\int_{\Omega} H(h(t,v,g))p^*(v,g) dvdg=-2\int_{\Omega} |\p_g h(t,v,g)|^2 p^* dvdg \leq 0.
\end{equation}
For convenience we introduce a fixed time $T>0$ 
and, for $k > T$, we define the sequence 
\begin{equation}
I_k:=2\int_{-T}^{T}dt\int_{\Omega}  |\p_gh_k(t,v,g)|^2p^* dvdg.
\end{equation}
Then, because the integral $\int_{0}^{\infty} \int_{\Omega}  |\p_gh(t,v,g)|^2p^* dvdg dt$, controlled by the initial entropy due to \eqref{entropy-key}, is finite, we have  
\begin{equation}
	 I_k \leq 2\int_{k-T}^{\infty}\int_{\Omega}  |\p_gh(t,v,g)|^2p^* dv dg dt
	 \rightarrow0,\quad k\rightarrow\infty.
\end{equation}
By the weak convergence and the convexity, we find that $\p_g h_\infty\in L^2\left( (-T,T)\times \Omega; p^*dtdvdg\right) $ and
\[
\int_{-T}^{T} \int_{\Omega}  |\p_g h_\infty (t,v,g)|^2p^* dv dg dt \leq \liminf_{k \to \infty} I_k =0.
\]
Thus we conclude that $h_\infty$ does not depend on $g$, which allows us to denote  $h_\infty:= F (t,v)$.
\\[2pt]
{\em Step 2. We prove $h_{\infty}\equiv 1$.}
 Since $h_{\infty}$  also satisfies Equation~\eqref{eq:h} in the distribution sense, we get, for a.e. $g>0$,
\begin{equation}\label{tmp}
	\p_t F(t,v)+(-g_Lv+g(V_E-v))\p_v F(t,v)=0,\quad t\in (-T,T), \; v \in (0, V_F).
\end{equation}
Then we can pick some $g_1>g_2>0$ in\eqref{tmp}, subtract one equation from the other, and get $\p_v F(t,v)=\p_t F(t,v)=0$. Therefore $F$ is constant and by the conservation of mass (preserved in the weak limit with the test function $1$ and using~\eqref{est:1}), we get $h_{\infty}\equiv 1$. 
\\[2pt]
{\em Step 3. Time compactness.}
Consider $q(t,v,g)=\partial_t p(t,v,g)$. Because the coefficients of Equation~\eqref{eq:iftdp} are independent of time for the case at hand, $q$ satisfies the same equation, and using the assumption on the initial data $p^0$, we conclude that $|q^0| \leq C^1_+ p^*$. Therefore, from the relative entropy, we infer that 
\[
| \partial_t p(t)|= |q(t)| \leq C^1_+ p^*.
\]
From this, we conclude that $ \partial_t h$ is bounded, and this provides us with the time compactness both for $h_k$ and $p_k$.
\\[2pt]
{\em Step 4. Compactness in $v$.}
We rewrite the equation for $h$, \eqref{eq:h}, as
\[
\p_v(J_vh)=[\p_{gg}h+\f{2}{p^*}\p_gp^*\p_gh]-\p_th - \left(g_{\rm in} - g\right)\p_gh + h \p_vJ_v, \qquad t > 0, \; g >0, \; v \in (0, V_F).
\]  Then together with the compactness of $h_k$ in $t$ and $g$, by the Lions--Aubin Lemma we deduce the (local) compactness of $J_vh_k$.
\\[1pt]
{\em Step 5. Convergence.} Thanks to the local compactness of $J_vh_k$ in all directions $g,\,t$ and $v$, we can extract a strongly convergent subsequence in $L^2_{loc}$. And the limit must be $J_vh_{\infty}$ owing to the unique weak limit of $h_k$, established in Step 2. Hence, the full sequence $J_vh_k$ converges to $J_vh_{\infty}$ in $L^2_{loc}$. This implies the convergence of $h_k$ in $L^2_{loc}$, since $J_v=-g_Lv+g(V_E-v)$ vanishes only on a zero measure set and $h_k$ has the $L^{\infty}$ bound as in \eqref{hinfini}.

Finally, using the $L^{\infty}$ bound of $h$ in~\eqref{hinfini}, we improve the convergence of $h_k$ to all $L^q_{loc}$. And via the $L^{\infty}$ bound of $p^*$ and the interior smoothness of $p^*$ we also obtain the $L^{q}_{loc}$ convergence of $p_k$ to $p^*$.
\end{proof}\qed

\subsection{Proof of Theorem \ref{thm2}}

We can improve the convergence in Proposition \ref{prop5} to prove Theorem \ref{thm2}.

Let us denote the entropy by
\[
G(h):=\int_{\Omega} H(h(v,g))p^*(v,g)dvdg,
\] where we take $H(h)=(h-1)^2$ as in the proof of Proposition \ref{prop5}. Then, from the entropy dissipation~\eqref{entropy-key},  we know $G(h(t))$ is decreasing in time. We shall show that $\lim_{t\rightarrow+\infty}G(h(t))=0$, hence proving Theorem~\ref{thm2}.

\begin{proof}
{\em Step 1. Integrated in time convergence.} First we assume the initial data satisfies the assumption \eqref{as:init-h} in Proposition \ref{prop5}. Therefore by Proposition \ref{prop5}, we obtain that $h_k(t,k,g)=h(t+k,v,g)$ converges to $h_{\infty}\equiv 1$ in $L^2_{loc}$, as a function of $t,v,g$. For some fixed $T>0$, the local convergence of $h_k$, together with its $L^{\infty}$ bound \eqref{hinfini} and the integrability of $p^*$, implies
\[
\begin{aligned}
    \int_{k}^{k+T}G(h(t))dt&=\int_{0}^{T}G(h_k(t))dt
    \\&=\int_{0}^{T}\int_{\Omega}H(h_k(t,v,g))p^*(v,g)dtdvdg\rightarrow 0,\qquad k\rightarrow\infty.
\end{aligned}
\]
\\[1pt]
{\em Step 2. Pointwise convergence.} Thanks to that $G(h(t))$ is decreasing in time, we deduce for $t\geq T$
\[
0\leq G(h(t))\leq \frac{1}{T}\int_{t-T}^t G(h(s))ds\rightarrow0,\qquad t\rightarrow\infty.
\]
Therefore Theorem \ref{thm2} is proved, provided that the initial data additionally satisfies \eqref{as:init-h}.\\[1pt]
{\em Step 3. General initial data.} For two solutions $h$ and $\bar{h}$, we have
\beq\label{eq:tria}
\bigl[G(h(t))\bigr]^{1/2}\leq \bigl[G(\bar{h}(t))\bigr]^{1/2}+\left[\int_{\Omega} (h(t,v,g)-\bar{h}(t,v,g))^2p^*(v,g)dvdg\right]^{1/2},
\eeq which follows from the triangle inequality for the weighted norm $L^2(p^*dvdg)$. Thanks to the linearity, $h-\bar{h}$ also satisfies \eqref{eq:h}. By taking $H(h)=h^2$ in the entropy dissipation \eqref{entropy-key}, we get that the last term in \eqref{eq:tria} is also decreasing in time. In addition, it is clear that when the initial value is smooth and compactly supported, \eqref{as:init-h} is naturally satisfied due to the positivity of $p^*$. Hence, by a density argument we conclude the case with general initial data satisfying $p^0 \leq C^0_+ p^*$.
\end{proof}\qed

\begin{remark}
Combining Theorem \ref{thm2} with the $L^{\infty}$ bound of $h$ \eqref{hinfini}, we get the convergence in all weighted norm $L^q(p^*dvdg)$ for $1\leq q<+\infty$.
\end{remark}
\begin{remark}
Notice that the arguments developed here improve the convergence result to a strong sense, which can be used in other models, e.g. \cite{fu2021fokker}.
\end{remark}

\section{Conclusion and perspectives} \label{sec:cp}

We have extended the methods in \cite{PeSa} in order to prove~$L^\infty$ bounds on the stationary solution of the linear voltage-conductance model for neuronal networks arising in the visual cortex. We have also established the long term convergence of the solution of the evolution problem to this steady state, thanks to the method of relative entropy. 
\\

Several questions remain open. A major question is to prove the exponential convergence, a route being to adapt the methods used for the kinetic Fokker-Planck equation (see the introduction). Also, in our theory, we took for granted the positivity and regularity of the stationary solution, which follows from the hypoellipticity of the underlying operator. In view of the numerous possible asymptotic problems of interest, it might be useful to prove quantitative positivity estimates.  
\\

Concerning the stationary problem, another viewpoint is to see the voltage $v$ as a time variable, and treat Equation~\eqref{eq:cif} as a time-evolution problem. In this way, it resembles a forward-backward parabolic equation in literature \cite{paronetto2004existence,paronetto2020further}. However, the unique boundary condition in $v$ of this model \eqref{eq:cif} has not been treated in this line of literature, as far as we know.
\\


Besides, two asymptotic limits for the voltage-conductance model are of interests: the vanishing noise limit $a\rightarrow0^+$ and the fast conductance limit $\sigma_E\rightarrow0^+$. Therefore, it is interesting to reexamine the estimates for the steady problem \eqref{eq:cif}, to see its dependence on $a$ and $\sigma_E$. However, a direct investigation shows that the bounds we have derived are not uniform for $a\rightarrow0^+$ or $\sigma_E\rightarrow0^+$. Therefore, substantial work is still needed to justify the asymptotic limits of such a kinetic model. 

\section*{Acknowledgment.} BP has received funding from the European Research Council (ERC) under the European Union's Horizon 2020 research and innovation programme (grant agreement No 740623). DS has received support from ANR ChaMaNe No: ANR-19-CE40-0024. ZZ is supported by the National Key R\&D Program of China, Project Number 2021YFA1001200, and the NSFC, grant Number 12171013. XD is partially supported by The Elite Program of
Computational and Applied Mathematics for PhD Candidates in Peking University. XD thanks Zhifei Zhang for helpful discussions.


%
%
%

\end{document}